\documentclass{amsart}

\usepackage{amsmath}
\usepackage{amsthm}
\usepackage{amsfonts}
\usepackage{amssymb}

\usepackage{graphics}
\usepackage{epsfig}
\usepackage{color}
\usepackage{verbatim}

\newtheorem{theorem}{Theorem}%[section]
\newtheorem{lemma}[theorem]{Lemma}

\theoremstyle{remark}
\newtheorem*{remark}{Remark}

\newtheorem*{Organization}{Organization}

\newtheorem*{question*}{Question}

\numberwithin{equation}{section}

\def\R{{\mathbb R}}

\thanks{This work was supported by JSPS Kakenhi grant numbers 18KK0073, 19H00644 and 19H01796 (Bez), 
16K05191 (Machihara), and 18KK0073, 19H00644 (Ozawa). }

\begin{document}

%\title{A quantitative improvement of the Rellich inequality}
\title{Revisiting the Rellich inequality}
\author{Neal Bez}
\address{Department of Mathematics, Faculty of Science, Saitama
University, Saitama 338-8570, Japan}
\email{nealbez@mail.saitama-u.ac.jp}
\author{Shuji Machihara}
\address{Department of Mathematics, Faculty of Science, Saitama
University, Saitama 338-8570, Japan}
\email{machihara@rimath.saitama-u.ac.jp}
\author{Tohru Ozawa}
\address{Department of Applied Physics, Waseda University, Tokyo 169-8555, Japan}
\email{txozawa@waseda.jp}

\maketitle

\begin{abstract}
We revisit the Rellich inequality from the viewpoint of isolating the contributions from radial and spherical derivatives. This naturally leads to a comparison of the norms of the radial Laplacian and Laplace--Beltrami operators with the standard Laplacian. In the case of the Laplace--Beltrami operator, the three-dimensional case is the most subtle and here we improve a result of Evans and Lewis by identifying the best constant. Our arguments build on certain identities recently established by Wadade and the second and third authors, along with use of spherical harmonics.
\end{abstract}

\section{Introduction and main results}
We will consider the following Rellich inequality
\begin{align}\label{Rellich1}
\left(\frac{n(n-4)}{4}\right)^2\left\|\frac{f}{|x|^2}\right\|^2_2\le\|\Delta f\|^2_2,
\end{align}
where $\|\cdot\|_2$ denotes the $L^2$-norm on $\R^n$ with respect to Lebesgue measure, and $\Delta$ denotes the standard Laplacian on $\R^n$. This inequality is a ubiquitous object in the theory of partial differential equations and spectral theory. 
The reader may refer to \cite{a-s, ca, da-h, e-l, g-g-m, gh-m, RS, sch, t-z, ya} and references therein for further details, applications and wider perspectives. 

If one wishes to treat general functions $f$ in the Sobolev space $H^2(\R^n)$, it is necessary to restrict the dimension of the ambient space to $n\ge5$, since the norm on the left-hand side of \eqref{Rellich1} will be infinite for any continuous function with $f(0)\not=0$ when $n \leq 4$. Rellich \cite{Rellich, RellichBook} proved that \eqref{Rellich1} holds for any $f\in C^{\infty}_0(\R^n\backslash\{0\})$ when $n\ge3$, and when $n = 2$, in order for the inequality to hold, we must additionally assume
\begin{align} \label{e:P1=0}
\int_0^{2\pi}f(r\cos\theta,r\sin\theta)\cos\theta \, \mathrm{d}\theta
=\int_0^{2\pi}f(r\cos\theta,r\sin\theta)\sin\theta \, \mathrm{d}\theta=0.
\end{align}
In fact, even if we replace the constant on the left-hand side by any positive constant, \eqref{Rellich1} fails when $n = 2$ if we admit all $f\in C^{\infty}_0(\R^2\backslash\{0\})$.

A fresh perspective on the Rellich inequality \eqref{Rellich1} based on certain identities was recently brought to light by Wadade and the second and third authors in \cite{m-o-w4}. Here we continue to develop this perspective by examining the contribution arising from radial derivatives and spherical derivatives. It is well-known that the Laplacian decomposes as $\Delta=\Delta_r+\frac{1}{|x|^2}\Delta_{\mathbb{S}^{n-1}}$ where $\Delta_r$ denotes 
the radial Laplacian and $\Delta_{\mathbb{S}^{n-1}}$ denotes the Laplace--Beltrami operator. The radial Laplacian is given by
\[
\Delta_r f = \partial_r^2f + \frac{n-1}{|x|}\partial_r f,
\]
where $\partial_r$ is the radial derivative given by $\partial_r =\frac{x}{|x|}\cdot\nabla$. Also, we introduce the spherical-type derivative $L_j=\partial_j-\frac{x_j}{|x|}\partial_r$ for each $j=1,\ldots,n$, in which case   
we have the relation
\begin{equation}  \label{e:sumL_j2}
\Delta_{\mathbb{S}^{n-1}}=|x|^2\sum_{j=1}^nL_j^2.
\end{equation}
The first of our key identities is the following and, roughly speaking, may be viewed as a means of capturing the radial and spherical contributions to both sides of \eqref{Rellich1}. In the following statement, and throughout the paper, we use the notation
\[
\mathbf{R}_{n} = \frac{n(n-4)}{4}
\]
associated with the Rellich inequality \eqref{Rellich1}.
\begin{theorem} \label{t:MOWv2}
For $n \geq 2$ and $f \in C_0^{\infty}(\R^n\backslash\{0\})$ we have
\begin{align} \label{e:MOW_left}
\mathbf{R}_{n}^2 \bigg\|\frac{f}{|x|^2}\bigg\|_{2}^2 &= \|\Delta_r f \|_{2}^2 -  
\bigg\| \Delta_r f + \mathbf{R}_{n} \frac{f}{|x|^2}   \bigg\|_{2}^2  - 2\mathbf{R}_{n} \bigg\| \frac{f_*}{|x|}  \bigg\|_{2}^2 \\
 \label{e:MOW_leftv2}
& = \|\Delta_r f \|_{2}^2 - \Big(1+\frac12\mathbf{R}_n\Big) 
\bigg\| \Delta_r f + \mathbf{R}_{n} \frac{f}{|x|^2}   \bigg\|_{2}^2 + \mathbf{R}_{n} \|f_{\#}\|_{2}^2
\end{align}
and
\begin{align} \label{e:MOW_right}
\|\Delta f\|_{2}^2 = \|\Delta_r f \|_{2}^2 + \bigg\| \sum_{j=1}^n L_j^2f \bigg\|_{2}^2 + 2\mathbf{R}_{n} \sum_{j=1}^n \bigg\| \frac{L_jf}{|x|}  \bigg\|_{2}^2 + 2\bigg\langle  -\sum_{j=1}^n L_j^2 f_*, f_* \bigg\rangle,
\end{align}
where $f_* = \partial_r f  + \frac{n-4}{2} \frac{f}{|x|}$, $f_{\#} = \partial_r^2f+(n-3)\frac{\partial_rf}{|x|}
+ \frac{(n-4)^2}{4} \frac{f}{|x|^2} $.
\end{theorem}
The identity \eqref{e:MOW_left} was reported in \cite[Theorem 1.1]{m-o-w4} for $n \geq 5$. Since $\mathbf{R}_4  = 0$, it is immediate that it extends to $n = 4$. Here we shall verify the cases $n = 2$ and $n = 3$; although the argument in \cite{m-o-w4} extends verbatim when $n = 2$, a certain amount of care is needed to check the case $n = 3$. The identity \eqref{e:MOW_leftv2} is new and it has the advantage that for $n = 2,3$, \emph{both} coefficients $-(1 + \frac{1}{2}\mathbf{R}_n)$ and $\mathbf{R}_n$ are negative. An identity similar to \eqref{e:MOW_right} was established in \cite[Theorem 1.2]{m-o-w4} (for $n \geq 5$), the only difference being the final term on the right-hand side of \eqref{e:MOW_right}. As we shall see, the form we have presented in Theorem \ref{t:MOWv2} is more convenient for our purposes in the present paper (when handling terms in a decomposition of $f$ into spherical harmonics), and we also emphasise that $-\sum_{j=1}^n L_j^2$ is a non-negative operator on $L^2(\mathbb{R}^n)$ (see the forthcoming discussion on spherical harmonics).

Regarding the second and third terms on the right-hand side of \eqref{e:MOW_right}, we have the identities in the forthcoming Theorem \ref{t:2nd3rdterms} in terms of the decomposition of $f$ into spherical harmonics. The statement is written is terms of the orthogonal projection operator $P_k$ onto the closed subspace $\mathcal{H}_k(\R^n)$ of $L^2(\R^n)$ spanned by spherical harmonics of 
order $k$ multiplied by radial functions. When $k = 0$, we will often abbreviate $P_0$ as $P$, and we note this operator is explicitly given by
\[
Pf(x) = \frac{1}{\sigma(\mathbb{S}^{n-1})} \int_{\mathbb{S}^{n-1}} f(|x|\omega) \, \mathrm{d}\sigma(\omega). 
\]
The spaces $\mathcal{H}_k(\R^n)$ give rise the standard spherical harmonic decomposition as
\begin{align*}
L^2(\R^n)=\bigoplus_{k=0}^{\infty}\mathcal{H}_k(\R^n)
\end{align*}
and the spherical harmonics of degree $k$ are well 
known to be eigenfunctions of the Laplace--Beltrami operator with eigenvalue $-\mu_k$, where 
\[
\mu_k := k(k+n-2);
\] 
that is
\begin{align} \label{e:eigen}
\Delta_{\mathbb{S}^{n-1}}P_k=-\mu_kP_k
\end{align}
holds for each $k \geq 0$. 
\begin{theorem} \label{t:2nd3rdterms} 
For $n \geq 2$ and $f \in C_0^{\infty}(\R^n\backslash\{0\})$, we have
\begin{align}\label{e:2nd3rdterms}
\begin{aligned}
% \bigg\|\sum_{j=1}^n L_j^2 P^\perp f \bigg\|_{2}^2-\lambda\sum_{j=1}^n  \bigg\|\frac{L_j P^\perp f}{|x|} \bigg\|_{2}^2  =\sum_{k = 1}^\infty\mu_k(\mu_k-\lambda)\bigg\| \frac{P_kP^\perp f}{|x|^2} \bigg\|_{2}^2.
\bigg\|\sum_{j=1}^n L_j^2 f \bigg\|_{2}^2=\sum_{k = 1}^\infty\mu_k^2\bigg\| \frac{P_k f}{|x|^2} \bigg\|_{2}^2, 
\qquad \sum_{j=1}^n  \bigg\|\frac{L_j f}{|x|} \bigg\|_{2}^2  =\sum_{k = 1}^\infty\mu_k\bigg\| \frac{P_k f}{|x|^2} \bigg\|_{2}^2.
\end{aligned}
\end{align}
\end{theorem}

As we shall see, Theorems \ref{t:MOWv2} and \ref{t:2nd3rdterms} lead to a number of interesting consequences, some of which are new, and some of which provide new and simpler perspectives on known results.

Firstly, we consider a refinement of the Rellich inequality \eqref{Rellich1} in which the Laplacian is replaced by the radial Laplacian. This naturally leads us to also consider a comparison of the size of the norm of the radial Laplacian and the Laplacian. We summarize matters as follows.
\begin{theorem}\label{t:radial}
Let $n \geq 2$. Then
\begin{align} \label{e:radialRellich}
\mathbf{R}_{n}^2 \bigg\|\frac{f}{|x|^2}\bigg\|_{2}^2 \leq \|\Delta_r f\|_2^2
\end{align}
holds for all $f \in C^\infty_0(\mathbb{R}^n \setminus \{0\})$. If $n \geq 3$, then
\begin{equation} \label{e:radialcomparison}
\|\Delta_r f\|_2\le C\|\Delta f\|_2
\end{equation}
holds for all $f \in C^\infty_0(\mathbb{R}^n \setminus \{0\})$ with $C = 1$, and fails for any constant $C$ when $n = 2$. Furthermore, the constants in \eqref{e:radialRellich} and \eqref{e:radialcomparison} are best possible.
\end{theorem}
The inequality \eqref{e:radialRellich} can be found already in work of Evans--Lewis \cite[Eq. (2.13)]{e-l}. We remark that it is possible to quickly derive \eqref{e:radialRellich} from \eqref{e:MOW_left} (for $n \geq 4$) and \eqref{e:MOW_leftv2} (for $n = 2,3$). The inequality \eqref{e:radialcomparison} with $C = 1$ follows immediately from \eqref{e:MOW_right} for $n \geq 4$ since $\mathbf{R}_n \geq 0$. When $n = 2$, \eqref{e:radialcomparison} must fail for any $C < \infty$ since otherwise we could combine with \eqref{e:radialRellich} and obtain an inequality of the form \eqref{Rellich1} with some positive constant on the left-hand side; as we have already pointed out, this is impossible. The case $n = 3$ is slightly more subtle since $\mathbf{R}_3 < 0$, but we may use Theorem \ref{t:2nd3rdterms} to see that
\begin{align*}
\begin{aligned}
 \bigg\|\sum_{j=1}^3 L_j^2 f \bigg\|_{2}^2 + 2\mathbf{R}_3 \sum_{j=1}^3  \bigg\|\frac{L_j f}{|x|} \bigg\|_{2}^2  
=\sum_{k = 1}^\infty \big(\mu_k( \mu_k-\tfrac{3}{2})\big) \bigg\| \frac{P_k f}{|x|^2} \bigg\|_{2}^2 \geq 0,
\end{aligned}
\end{align*}
where we have used that $\mu_k = k(k+1) \geq 2$ for $k \geq 1$.

For completeness we also note that the optimality of the constant in \eqref{e:radialRellich} can be shown by arguing as in \cite{Cazacu} using the family of functions $f_\varepsilon(x) = |x|^{\frac{4-n}{2} + \frac{\varepsilon}{2}}\chi(|x|)$. Here $\chi$ is infinitely smooth, $\chi(0) = 1$ and $\chi(r)$ vanishes for $r \geq 1$. Then, as $\varepsilon \to 0$, we have $\varepsilon \| \Delta_r f\|_2^2 \to \mathbf{R}_n^2$ and $\varepsilon \|\frac{f}{|x|^2}\|_{2}^2 \to 1$. Also, the constant $C = 1$ in \eqref{e:radialcomparison} clearly cannot be improved since equality obviously holds for radial $f$.

Next we present a direct analogue of Theorem \ref{t:radial} for spherical derivatives.
\begin{theorem}\label{t:spherical}
Let $n \geq 2$. Then
\begin{align} \label{e:sphericalRellich}
(n-1)^2\bigg\| \frac{f}{|x|^2} \bigg\|_2^2 \leq \bigg\| \sum_{j=1}^n L_j^2  f \bigg\|_2^2
\end{align}
holds for all $f \in C^\infty_0(\mathbb{R}^n \setminus \{0\})$ such that $Pf = 0$. If $n \geq 3$, then
\begin{equation} \label{e:sphericalcomparison}
\bigg\| \sum_{j=1}^n L_j^2f \bigg\|_{2}^2 \leq \mathbf{C}_n \|\Delta f\|_{2}^2
\end{equation}
holds for all $f \in C^\infty_0(\mathbb{R}^n \setminus \{0\})$, with $\mathbf{C}_n = 1$ when $n \geq 4$, with $\mathbf{C}_3 = \frac{64}{25}$ when $n = 3$, and the inequality fails for any $\mathbf{C}_2 < \infty$ when $n = 2$. Furthermore, the constants in \eqref{e:sphericalRellich} and \eqref{e:sphericalcomparison} are best possible.
\end{theorem}
The inequality \eqref{e:sphericalRellich} follows immediately from Theorem \ref{t:2nd3rdterms} since we have
\[
\bigg\| \frac{f}{|x|^2}\bigg\|_2^2 = \sum_{k = 1}^\infty \bigg\| \frac{P_k f}{|x|^2} \bigg\|_{2}^2
\]
whenever $Pf = 0$, and $\mu_k \geq n - 1$ for $k \geq 1$. It is also clear from this argument that the constant is best possible since equality holds in \eqref{e:sphericalRellich} for functions in $\mathcal{H}_1(\mathbb{R}^n)$.

The inequality in \eqref{e:sphericalcomparison} was considered by Evans--Lewis in \cite[Corollary 1]{e-l}\footnote{In fact, an abstract version of \eqref{e:sphericalcomparison} was considered in \cite{e-l}, and we shall identify the best constants in such a setting in Section \ref{section:abstract}.}. An examination of their proof yields the explicit constants $\mathbf{C}_n = 1$ when $n \geq 4$, and $\mathbf{C}_3 = 4$ (see Section \ref{section:alternative} for details of the three-dimensional case), but there appears to be no discussion on the best constant in \cite{e-l}. Note that the fact that \eqref{e:sphericalcomparison} holds with  $\mathbf{C}_n = 1$ when $n \geq 4$ follows immediately from \eqref{e:MOW_right}. The case $n = 3$ is much more subtle and we shall see that it is possible to obtain the best constant by making use of \eqref{e:MOW_right} and \eqref{e:2nd3rdterms}.

Finally, in a slightly different direction, we present a refinement of the Rellich inequality \eqref{Rellich1} in which an additional term is added to the left-hand side. For the statement, we introduce the constant 
\[
\widetilde{\mathbf{R}}_{n}  = (n-1)(n-1 + 2\mathbf{R}_{n})
\] 
and write $P^{\perp}=I-P$ for the projection operator onto the orthogonal complement of $\mathcal{H}_0(\mathbb{R}^n)$ (closed subspace of $L^2(\mathbb{R}^n)$ consisting of radially symmetric functions).
\begin{theorem} \label{t:LHSbiggerRellich}
Let $n \geq 2$. Then
\begin{equation} \label{e:remainderR}
\mathbf{R}_{n}^2 \bigg\|\frac{f}{|x|^2}\bigg\|_{2}^2 + \widetilde{\mathbf{R}}_{n}  \bigg\|\frac{P^\perp f}{|x|^2}\bigg\|_{2}^2 \leq \|\Delta f\|_{2}^2
\end{equation}
holds for all $f \in C^\infty_0(\mathbb{R}^n \setminus \{0\})$.
\end{theorem}
Since $\widetilde{\mathbf{R}}_{n} \geq 0$ for $n \geq 3$ we immediately recover the classical Rellich inequality \eqref{Rellich1} from \eqref{e:remainderR}. If we instead apply \eqref{e:remainderR} for input functions satisfying $Pf = 0$ (i.e. those $f$ which are orthogonal to the subspace of radially symmetric functions), then we obtain 
\begin{equation} \label{e:orthogonalR}
(\mathbf{R}_{n} + n-1)^2\bigg\|\frac{f}{|x|^2}\bigg\|_{2}^2 \leq \|\Delta f\|_{2}^2\qquad\text{for}\quad Pf=0, 
\end{equation}
and thus a quantitative improvement in the size of the constant for $n\ge 3$. The improved form of the Rellich inequality \eqref{e:orthogonalR} subject to the constraint $Pf = 0$ is not new and follows from the abstract result in \cite[Theorem 1]{e-l} (see also work by Caldiroli--Musina \cite{CM}); here we give a different perspective and obtain \eqref{e:orthogonalR} from Theorems \ref{t:MOWv2} and \ref{t:2nd3rdterms}, and use this perspective to give a relatively simple proof that there are no (non-trivial) maximizers for the inequality (see Section \ref{section:nonexistence}). The fact that the inequality \eqref{e:orthogonalR} extends to the Sobolev space $H^2(\mathbb{R}^n)$ for $n \geq 3$ (with $Pf = 0$) requires an additional density argument and since we were unable to find this anywhere, we have included details in Section \ref{section:density} of the present paper. We also remark that, conversely, \eqref{e:remainderR} can be derived from \eqref{Rellich1} and \eqref{e:orthogonalR}\footnote{Indeed, by orthogonality and applications of \eqref{Rellich1} and \eqref{e:orthogonalR} we have
\begin{align*}
\mathbf{R}_n^2\left\|\frac{f}{|\cdot|^2}\right\|^2
+\widetilde{\mathbf{R}}_n\left\|\frac{P^{\perp}f}{|\cdot|^2}\right\|^2 =\mathbf{R}_n^2\left\|\frac{Pf}{|\cdot|^2}\right\|^2+(\mathbf{R}_n + n-1)^2
\left\|\frac{P^{\perp}f}{|\cdot|^2}\right\|^2 \le\|\Delta Pf\|^2+\|\Delta P^{\perp}f\|^2=\|\Delta f\|^2.
\end{align*}
In the last step we used the fact that $P$ commutes with $\Delta$. The observation here was pointed out to us on an earlier version of the present paper.}.

As an application which ties things discussed in the present paper together, we shall see that it is possible to give another proof of \eqref{e:sphericalcomparison} with the best constant in three dimensions by using \eqref{e:orthogonalR} (see Section \ref{section:alternative}). 

The observation that the Hardy inequality on $L^2(\mathbb{R}^n)$ may be improved by constraining the class of input functions to $Pf = 0$ may be found in work of Birman--Laptev \cite{BL} and Ekholm--Frank \cite{e-f} in the context of the spectral theory of Schr\"odinger operators (see also \cite{b-m-o, DLP}). The current paper may be viewed as a natural continuation of \cite{b-m-o} in the context of the Rellich inequality, not only in terms of the result but also in terms of our approach.

\begin{Organization}
In Section \ref{section:identities}, we prove the identities in Theorems \ref{t:MOWv2} and \ref{t:2nd3rdterms}. The claims in Theorem \ref{t:radial} have already been justified above (in the discussion after the statement of Theorem \ref{t:radial}). In Section \ref{section:inequalities} we prove Theorems \ref{t:spherical} and \ref{t:LHSbiggerRellich}. 

In Section \ref{section:final}, we provide various further results. In Section \ref{section:alternative} we give an alternative proof of \eqref{e:sphericalcomparison} in three dimensions via \eqref{e:orthogonalR} and using ideas from \cite{e-l}. This argument extends to the abstract setting considered in \cite[Corollary 1]{e-l} and we prove a generalization of Theorem \ref{t:spherical} in Section \ref{section:abstract}. In Section \ref{section:density} we give details of a density argument which extends \eqref{e:orthogonalR} to $L^2$-based Sobolev spaces for three dimensions and higher, and finally in Section \ref{section:nonexistence} we also show how the approach taken in the current paper facilitates a straightforward route to establishing the lack of non-trivial cases of equality in \eqref{e:orthogonalR}. 
\end{Organization}

\section{Proofs of Theorems \ref{t:MOWv2} and \ref{t:2nd3rdterms}} \label{section:identities}

We write
\[
\mathbf{H}_{n} = \frac{n-2}{2}
\]
for the constant associated with the Hardy inequality on $L^2(\mathbb{R}^n)$.

\begin{proof}[Proof of \eqref{e:MOW_left}]
As stated just after Theorem \ref{t:MOWv2}, the argument in \cite[Theorem 1.1]{m-o-w4} is valid as it stands for $n \geq 2$ and $n \neq 3$ so we refer the reader to that paper for the details. Here we give details for the case $n = 3$; this case requires some clarification since the proof of \cite[Theorem 1.1]{m-o-w4} proceeds initially using integration by parts to write
	\begin{equation*}
	\bigg\| \frac{f}{|x|^2} \bigg\|_{2}^2 = \frac{2}{(n-3)(n-4)} \bigg(\bigg\| \frac{\partial_rf}{|x|} \bigg\|_{2}^2 + \mathrm{Re} \bigg\langle \frac{f}{|x|^2} , \partial_r^2 f \bigg\rangle \bigg).
	\end{equation*}
In our current case $n= 3$, this is valid in the sense that
\begin{equation} \label{e:n=3special}
\bigg\| \frac{\partial_rf}{|x|} \bigg\|_{2}^2 = - \mathrm{Re} \bigg\langle \frac{f}{|x|^2} , \partial_r^2 f \bigg\rangle
\end{equation}
as can be verified by an integration by parts. Thus, we may proceed along the same line of reasoning as in the proof of \cite[Theorem 1.1]{m-o-w4} as follows. Firstly, we begin with the equality
\begin{equation} \label{e:Hardyequality}
	\mathbf{H}_{3}^2 \bigg\| \frac{f}{|x|} \bigg\|_{2}^2 = \|\partial_r f\|_{2}^2   - \bigg\|\partial_r f + \mathbf{H}_{3} \frac{f}{|x|} \bigg\|_{2}^2  
\end{equation}
associated with the Hardy inequality on $\mathbb{R}^3$. To see \eqref{e:Hardyequality}, note that an integration by parts gives
	\begin{equation*}
	\bigg\| \frac{f}{|x|} \bigg\|_{2}^2 = -\frac{1}{\mathbf{H}_{3}} \mathrm{Re} \bigg\langle \frac{f}{|x|} , \partial_r f \bigg\rangle
	\end{equation*}
from which we obtain the desired equality by an application of the (easily verified) equivalence
\begin{equation} \label{e:MOW1.4}
\|u\|^2 = -c \mathrm{Re} \langle u,v \rangle + b \quad \Leftrightarrow \quad \frac{1}{c^2} \|u\|^2 = \|v\|^2 - \bigg\|v + \frac{1}{c}u\bigg\|^2 + \frac{2b}{c^2},
\end{equation}
which holds abstractly for $u,v$ in some Hilbert space, $b \in \mathbb{R}$ and $c \in \mathbb{R} \setminus \{0\}$.

Next, observe that
	\[
	\bigg\| \frac{\partial_rf}{|x|} \bigg\|_{2}^2 = \bigg\|\partial_r \frac{f}{|x|} \bigg\|_{2}^2 + \bigg\| \frac{f}{|x|^2} \bigg\|_{2}^2 + 2\mathrm{Re} \bigg\langle \partial_r \frac{f}{|x|},  \frac{f}{|x|^2} \bigg\rangle
	\]
	and \eqref{e:Hardyequality} implies
	\[
	\bigg\|\partial_r \frac{f}{|x|} \bigg\|_{2}^2 = \mathbf{H}_{3}^2 \bigg\| \frac{f}{|x|^2} \bigg\|_{2}^2 + \bigg\|\partial_r \frac{f}{|x|} + \mathbf{H}_{3} \frac{f}{|x|^2} \bigg\|_{2}^2.
	\]
	Also, a further integration by parts gives
	\[
	\mathrm{Re} \bigg\langle \partial_r \frac{f}{|x|},  \frac{f}{|x|^2} \bigg\rangle = -  \mathbf{H}_{3} \bigg\| \frac{f}{|x|^2} \bigg\|_{2}^2
	\]
	and so we obtain
	\[
	\bigg\| \frac{\partial_rf}{|x|} \bigg\|_{2}^2 = (\mathbf{H}_{3}- 1)^2 \bigg\| \frac{f}{|x|^2} \bigg\|_{2}^2 + \bigg\|\partial_r \frac{f}{|x|} + \mathbf{H}_{3} \frac{f}{|x|^2} \bigg\|_{2}^2. 
	\]
	Yet another integration by parts gives
	\[
	\mathrm{Re}\bigg \langle \frac{f}{|x|^2} , \partial_r^2 f \bigg\rangle  = \mathrm{Re} \bigg\langle \frac{f}{|x|^2} , \Delta_r f \bigg\rangle + 2(\mathbf{H}_{3}-1) \bigg\| \frac{f}{|x|^2} \bigg\|_{2}^2.
	\]
	Therefore, from \eqref{e:n=3special}, we get
	\[
	\bigg\| \frac{f}{|x|^2} \bigg\|_{2}^2 = -\frac{1}{\mathbf{R}_{3}} \mathrm{Re} \bigg\langle \frac{f}{|x|^2} , \Delta_r f \bigg\rangle - \frac{1}{\mathbf{R}_{3}} \bigg\|\partial_r \frac{f}{|x|} + \mathbf{H}_{3} \frac{f}{|x|^2} \bigg\|_{2}^2
	\]
	and hence, by \eqref{e:MOW1.4}, we may conclude \eqref{e:MOW_left} for $n=3$.
\end{proof}

\begin{proof}[Proof of \eqref{e:MOW_leftv2}]
For $a\in\R$, we start by observing that
\begin{align}\label{equality10}
\frac{(n-2-a)^2}{4}\left\|\frac{g}{|x|^{1+\frac{a}{2}}}\right\|_2^2
=\left\|\frac{\partial_rg}{|x|^{\frac{a}{2}}}\right\|_2^2
-\left\|\frac{n-2-a}{2}\frac{g}{|x|^{1+\frac{a}{2}}}+\frac{\partial_rg}{|x|^{\frac{a}{2}}}\right\|_2^2
\end{align}
holds for all $g\in C_0^{\infty}(\R^n\backslash\{0\})$. Indeed, this follows by expanding the second term on the right-hand side, and then using integration by parts to see that
\begin{align*}
2{\rm Re}\bigg\langle\frac{g}{|x|^{1+\frac{a}{2}}}, \frac{\partial_rg}{|x|^{\frac{a}{2}}}\bigg\rangle
=-(n-2-a)\left\|\frac{g}{|x|^{1+\frac{a}{2}}}\right\|^2.
\end{align*}
Now taking $a=n+2$ and $g$ given by
\begin{align*}
g(x)=\left(\frac{\partial_rf}{|x|}+(\mathbf{H}_n-1)\frac{f}{|x|^2}\right)|x|^{\frac{n+4}{2}}
\end{align*}
we see that
\begin{align*}
\partial_rg(x)&=\left(\Delta_rf+\mathbf{R}_n\frac{f}{|x|^2}\right)|x|^{\frac{n+2}{2}},
\end{align*}
and therefore \eqref{e:MOW_leftv2} follows from \eqref{e:MOW_left} and
 \eqref{equality10}.
\end{proof}

\begin{proof}[Proof of \eqref{e:MOW_right}]
It was shown in \cite[Theorem 1.2]{m-o-w4} that
\begin{align*} 
\|\Delta f\|_{2}^2 = \|\Delta_r f \|_{2}^2 + \bigg\| \sum_{j=1}^n L_j^2f \bigg\|_{2}^2 + 2\mathbf{R}_{n} \sum_{j=1}^n \bigg\| \frac{L_jf}{|x|}  \bigg\|_{2}^2 + 2\sum_{j=1}^n \bigg\|  \partial_r L_jf  + \mathbf{H}_n\frac{L_jf}{|x|} \bigg\|_{2}^2
\end{align*}
holds for $f\in C_0^{\infty}(\R^n\backslash\{0\})$ and $n \geq 5$, but in fact the proof works just as well for any $n \geq 2$. 

Now note that $L_j\partial_r = (\partial_r + \frac{1}{|x|})L_j$ and thus, since we also know that $L_j$ commutes with multiplication by radial functions, we may write
\begin{align*}
\sum_{j=1}^n \bigg\|  \partial_r L_jf  + \mathbf{H}_n\frac{L_jf}{|x|} \bigg\|_{2}^2 = \sum_{j=1}^n \bigg\|  L_j \bigg(\partial_r f  + (\mathbf{H}_n - 1) \frac{f}{2|x|} \bigg)\bigg\|_{2}^2.
\end{align*}
By using the integration by parts identity
\begin{equation} \label{e:Ljparts} 
\langle g,L_j h \rangle = - \langle L_j g, h \rangle + (n-1) \langle g, \frac{x_j}{|x|^2} h \rangle,
\end{equation}
and the fact that $\sum_{j=1}^n x_jL_j = 0$, we obtain \eqref{e:MOW_right}.
\end{proof}

\begin{proof}[Proof of \eqref{e:2nd3rdterms}]
Firstly, we have
	\begin{align*}
	\bigg\|\sum_{j=1}^n L_j^2 f \bigg\|_{2}^2 & 
	= \bigg\|\frac{1}{|x|^2} \Delta_{\mathbb{S}^{n-1}} f \bigg\|_{2}^2 \\
	& = \int_0^\infty \| (\Delta_{\mathbb{S}^{n-1}} f)(r(\cdot)) \|_{L^2(\mathbb{S}^{n-1})}^2 r^{n-5} \, \mathrm{d}r \\
	& = \sum_{k=1}^\infty \mu_k^2
	\int_0^\infty \| (P_k f)(r(\cdot)) \|_{L^2(\mathbb{S}^{n-1})}^2 r^{n-5} \, \mathrm{d}r
	\end{align*}
	using \eqref{e:sumL_j2} and the eigenfunction relation \eqref{e:eigen}.
	%, and $P^\perp P_k = P_kP^\perp$. 
	Therefore
	\begin{equation*} 
	\bigg\|\sum_{j=1}^n L_j^2 f \bigg\|_{2}^2 
	= \sum_{k = 1}^\infty \mu_k^2 \bigg\| \frac{P_k  f}{|x|^2} \bigg\|_{2}^2.
	\end{equation*}
	
	Next, observe that
	\begin{align*}
	\sum_{j=1}^n  \bigg\|\frac{L_j f}{|x|} \bigg\|_{2}^2 & 
	= - \bigg\langle f ,  \frac{1}{|x|^2} \sum_{j=1}^n L_j^2 f \bigg\rangle.
	\end{align*}
	To see this, we use the integration by parts formula \eqref{e:Ljparts}, 
the fact that $L_j$ communtes with multiplication by $|x|^{-2}$, and the identity $\sum_{j=1}^n x_j L_j = 0$. Therefore, by \eqref{e:sumL_j2},
	\begin{align*}
	\sum_{j=1}^n  \bigg\|\frac{L_j f}{|x|} \bigg\|_{2}^2 
	= - \bigg\langle f , \frac{1}{|x|^4} \Delta_{\mathbb{S}^{n-1}} f \bigg\rangle =  -\int_0^\infty \langle f(r(\cdot)) , \Delta_{\mathbb{S}^{n-1}} f(r(\cdot)) \rangle_{L^2(\mathbb{S}^{n-1})} r^{n-5} \, \mathrm{d}r \\
	\end{align*}
	and by orthogonality we get
	\begin{align*}
	\sum_{j=1}^n  \bigg\|\frac{L_j f}{|x|} \bigg\|_{2}^2 & 
	=  \sum_{k=0}^\infty \mu_k \int_0^\infty \|P_k f(r(\cdot))\|_{L^2(\mathbb{S}^{n-1})}^2 r^{n-5} \, \mathrm{d}r = \sum_{k=1}^\infty \mu_k \bigg\|\frac{P_k f}{|x|^2}\bigg\|_{2}^2
	\end{align*}
	as claimed.
\end{proof}

\section{Proofs of Theorems \ref{t:spherical} and \ref{t:LHSbiggerRellich}} \label{section:inequalities}
As preparation for the proof of \eqref{e:sphericalcomparison} in Theorem \ref{t:spherical}, we introduce the weighted one-dimensional version of the Hardy inequality (see, for example, \cite[Proposition 2.4]{c-p}).
\begin{lemma} \label{l:1DHardy}
Let $t \in \mathbb{R} \setminus \{1\}$. Then
\begin{equation} \label{e:1DHardy}
\int_0^\infty |g'(r)|^2 r^{t + 2} \, \mathrm{d}r \geq \bigg(\frac{t+1}{2}\bigg)^2 \int_0^\infty |g(r)|^2 r^{t} \, \mathrm{d}r 
\end{equation}
and the constant is best possible.
\end{lemma}
As a consequence, we have the weighted Rellich inequality 
\begin{equation} \label{e:1DRellich}
\int_0^\infty |g''(r)|^2 r^{t + 4} \, \mathrm{d}r \geq  \bigg(\frac{t+3}{2}\bigg)^2 \bigg(\frac{t+1}{2}\bigg)^2 \int_0^\infty |g(r)|^2 r^{t} \, \mathrm{d}r 
\end{equation}
for $t \neq -3,-1$, and it is also known that the constant here is best possible\footnote{To see this directly, one may argue as in \cite{Cazacu} and consider $g$ of the form $g_\varepsilon(r) = r^{-\frac{t + 1}{2} + \frac{\varepsilon}{2}}\chi(r)$ where $\chi$ is infinitely smooth, $\chi(0) = 1$ and $\chi(r)$ vanishes for $r \geq 1$. Then, as $\varepsilon \to 0$, we have $\varepsilon \int_0^\infty |g_\varepsilon(r)|^2 r^{t} \, \mathrm{d}r \to 1$ and $\varepsilon \int_0^\infty |g_\varepsilon'(r)|^2 r^{t + 2} \, \mathrm{d}r  \to (\frac{t+3}{2})^2(\frac{t+1}{2})^2$.}.

\begin{proof}[Proof of Theorem \ref{t:spherical}]
We have already justified \eqref{e:sphericalRellich} after the statement of Theorem \ref{t:spherical}, so here we prove \eqref{e:sphericalcomparison} and justify the claim that the derived constants are best possible.

First let $n \geq 4$. Since $\mathbf{R}_n \geq 0$ the inequality \eqref{e:sphericalcomparison} follows immediately from \eqref{e:MOW_right} with $\mathbf{C}_n = 1$. To see that this is optimal, test the inequality on $f$ of the form
\begin{equation} \label{e:Hkinput}
f(x)= g(|x|) Y_{k}(\tfrac{x}{|x|})
\end{equation}
where $Y_k$ is any spherical harmonic of degree $k$ with $L^2(\mathbb{S}^{n-1})$ norm equal to $1$. Then we easily obtain
\[
C_1 \mu_k^2 \leq \mathbf{C}_n(C_1 \mu_k^2 + C_2 \mu_k + C_3)
\]
where
\begin{align*}
C_1 & := \int_0^\infty \bigg|g''(r) + \frac{n-1}{r}g'(r)\bigg|^2 r^{n-1} \, \mathrm{d}r, \\
C_2 & := \int_0^\infty \bigg(g''(r) + \frac{n - 1}{r} g'(r)\bigg)\overline{g(r)} r^{n-3} \, \mathrm{d}r, \\ 
C_3 & := \int_0^\infty |g(r)|^2 r^{n-5} \, \mathrm{d}r 
\end{align*}
are constants (independent of $k$) which are all finite and $C_1 \neq 0$ by an appropriate choice of $g$. By taking $k \to \infty$, it is clear that $\mathbf{C}_n \geq 1$ is necessary for \eqref{e:sphericalcomparison} to hold.

The case $n = 2$ can also be handled in an easy manner by testing the inequality \eqref{e:sphericalcomparison}  on $f$ of the form \eqref{e:Hkinput} with $k = 1$. Indeed, since $-\Delta_{\mathbb{S}^1} Y_1 = Y_1$ we see that \eqref{e:sphericalcomparison} is equivalent to 
\[
\bigg\| \frac{f}{|x|^2} \bigg\|^2 \leq \mathbf{C}_2 \|\Delta f\|_{2}^2
\]
for such $f$. On the other hand, we know that the above Rellich inequality cannot hold for such $f$ with any constant $\mathbf{C}_2<\infty$ when $n = 2$ (since \eqref{e:P1=0} does not hold for such $f$).

Suppose now that $n = 3$ and $\mathbf{C}_3 = \frac{64}{25}$. To prove \eqref{e:sphericalcomparison}, by \eqref{e:MOW_right}, it is equivalent to prove
\begin{align*}
\frac{3}{2} \sum_{j=1}^3 \bigg\| \frac{L_jf}{|x|}  \bigg\|_{2}^2 \leq \|\Delta_r f \|_{2}^2 + & \bigg(1 - \frac{1}{\mathbf{C}_3}\bigg) \bigg\| \sum_{j=1}^3 L_j^2f \bigg\|_{2}^2 \\
& - 2\bigg\langle  \sum_{j=1}^3 L_j^2 \bigg(\partial_r f  - \frac{f}{2|x|} \bigg), \partial_r f  - \frac{f}{2|x|} \bigg\rangle. 
\end{align*}
For this, we use Theorem \ref{t:2nd3rdterms} to write the goal as
\begin{align*}
\frac{3}{2} \sum_{k=1}^\infty \mu_k \bigg\|\frac{P_k  f}{|x|^2}\bigg\|_{2}^2 \leq \|\Delta_r f \|_{2}^2  + & \bigg(1 - \frac{1}{\mathbf{C}_3}\bigg) \sum_{k=1}^\infty \mu_k^2 \bigg\|\frac{P_k  f}{|x|^2}\bigg\|_{2}^2 \\
& - 2\bigg\langle  \sum_{j=1}^3 L_j^2 \bigg(\partial_r f  - \frac{f}{2|x|} \bigg), \partial_r f  - \frac{f}{2|x|} \bigg\rangle.  
\end{align*}
For all $k \geq 2$ we have
$
\frac{3}{2}\mu_k \leq (1 - \frac{1}{\mathbf{C}_3})\mu_k^2
$
since $\mathbf{C}_3 > \frac{4}{3}$. So it suffices to establish\footnote{Note that this obviously holds with $\mathbf{C}_3 = 4$ and this recovers what can be obtained using the argument by Evans--Lewis \cite[Corollary 1]{e-l}.}
\begin{equation} \label{e:sphericalgoal}
\bigg(\frac{4}{\mathbf{C}_3} - 1\bigg) \bigg\|\frac{P_1  f}{|x|^2}\bigg\|_{2}^2 \leq \|\Delta_r f \|_{2}^2  - 2\bigg\langle  \sum_{j=1}^3 L_j^2 \bigg(\partial_r f  - \frac{f}{2|x|} \bigg), \partial_r f  - \frac{f}{2|x|} \bigg\rangle.  
\end{equation}

Expanding $f$ into spherical harmonics we may write
\begin{equation*} \label{e:decomposition}
f(x)= \sum_{k = 0}^\infty \sum_{m=1}^{N_k} g_{k,m}(|x|) Y_{k,m}(\tfrac{x}{|x|}), 
\end{equation*}
where $N_k := \mathrm{dim}(\mathcal{H}_k(\mathbb{R}^n))$, $\{Y_{k,m} : m=1,\ldots,N_k\}$ is an orthonormal basis of the space of spherical harmonics of degree $k$, and appropriate functions $g_{k,m}$, in which case \eqref{e:sphericalgoal} is equivalent to
\begin{align*}
\bigg(\frac{4}{\mathbf{C}_3} - 1\bigg)  \sum_{m=1}^{N_1} \int_0^\infty |g_{1,m}(r)|^2 \, \frac{\mathrm{d}r}{r^2} 
\leq &
\sum_{k = 0}^\infty \sum_{m=1}^{N_k}  \int_0^\infty \bigg|g_{k,m}''(r) + \frac{2}{r}g_{k,m}'(r)\bigg|^2 r^2 \, \mathrm{d}r \\
& + 2\sum_{k = 0}^\infty \sum_{m=1}^{N_k} \mu_k \int_0^\infty \bigg|g_{k,m}'(r) - \frac{1}{2r}g_{k,m}(r)\bigg|^2 r^2 \, \mathrm{d}r
\end{align*}
and so it suffices to prove
\begin{align*}
\bigg(\frac{4}{\mathbf{C}_3} - 1\bigg) \int_0^\infty |g(r)|^2 \, \frac{\mathrm{d}r}{r^2} 
\leq \int_0^\infty \bigg|g''(r) + \frac{2}{r}g'(r)\bigg|^2 r^2 \, \mathrm{d}r + 4 \int_0^\infty \bigg|g'(r) - \frac{1}{2r}g(r)\bigg|^2  \, \mathrm{d}r.
\end{align*}
By expanding the squares on the right-hand side, integration by parts on the cross terms and rearranging, we see that this is equivalent to
\begin{align} \label{e:sphericallaststep}
\frac{4}{\mathbf{C}_3} \int_0^\infty |g(r)|^2 \, \frac{\mathrm{d}r}{r^2} 
\leq & \int_0^\infty |g''(r)|^2 r^2 \, dr + 6 \int_0^\infty |g'(r)|^2  \, \mathrm{d}r.
\end{align}
By double use of \eqref{e:1DHardy}, we see that
\begin{align*}
\int_0^\infty |g''(r)|^2 r^2 \, \mathrm{d}r + 6 \int_0^\infty |g'(r)|^2  \, \mathrm{d}r & \geq \frac{25}{4} \int_0^\infty |g'(r)|^2  \, \mathrm{d}r  \\
& \geq \frac{25}{16} \int_0^\infty |g(r)|^2 \, \frac{\mathrm{d}r}{r^2},
\end{align*}
and hence, by the choice of $\mathbf{C}_3$ we obtain \eqref{e:sphericallaststep}, and thus \eqref{e:sphericalcomparison}.

To see that $\mathbf{C}_3 = \frac{64}{25}$ is the best constant, consider $f$ of the form \eqref{e:Hkinput} with $k = 1$. From the above argument, for such $f$, we see that \eqref{e:sphericalcomparison} is \emph{equivalent} to \eqref{e:sphericallaststep}, and so matters reduce to obtaining the optimal constant in \eqref{e:sphericallaststep}. One way to argue is to set $g(r) = r^\nu h(r)$ and note that by choosing $\nu = \sqrt{3}$ the inequality \eqref{e:sphericallaststep} is equivalent to
\begin{align*}
\bigg(\frac{4}{\mathbf{C}_3} + 6 \bigg) \int_0^\infty |h(r)|^2 r^{2\nu - 2} \, \mathrm{d}r
\leq \int_0^\infty |h''(r)|^2 r^{2\nu + 2} \, \mathrm{d}r. 
\end{align*}
Now we refer to the fact stated after Lemma \ref{l:1DHardy} that the best constant in \eqref{e:1DRellich} with $t = 2\nu - 2$ is $\frac{121}{16} = (\frac{4}{\mathbf{C}_3} + 6)$. From this, we may deduce the optimality of $\mathbf{C}_3 = \frac{64}{25}$ in \eqref{e:sphericalcomparison}.
\end{proof}

\begin{proof}[Proof of Theorem \ref{t:LHSbiggerRellich}]
From \eqref{e:2nd3rdterms} we have
\begin{align*}
 \bigg\|\sum_{j=1}^n L_j^2 f \bigg\|_{2}^2+ 2\mathbf{R}_{n} \sum_{j=1}^n  \bigg\|\frac{L_j f}{|x|} \bigg\|_{2}^2   = \widetilde{\mathbf{R}}_{n} \bigg\|\frac{P^\perp f}{|x|^2}\bigg\|_2^2 + \sum_{k = 1}^\infty \big(\mu_k( \mu_k + 2\mathbf{R}_{n}) - \widetilde{\mathbf{R}}_{n} \big) \bigg\| \frac{P_kP^\perp f}{|x|^2} \bigg\|_{2}^2
\end{align*}
and clearly $\mu_k( \mu_k + 2\mathbf{R}_{n}) - \widetilde{\mathbf{R}}_{n} \geq \mu_1( \mu_1 + 2\mathbf{R}_{n}) - \widetilde{\mathbf{R}}_{n}  = 0$ for all $k \geq 1$. By \eqref{e:MOW_left} and \eqref{e:MOW_right} we have
\begin{align}\label{mix-equal}
\begin{split}
& \|\Delta f\|_{2}^2 - \mathbf{R}_{n}^2 \bigg\|\frac{f}{|x|^2}\bigg\|_{2}^2 - \widetilde{\mathbf{R}}_{n} \bigg\|\frac{P^\perp f}{|x|^2}\bigg\|_{2}^2-\sum_{k = 1}^\infty \big(\mu_k( \mu_k + 2\mathbf{R}_{n}) - \widetilde{\mathbf{R}}_{n} \big) \bigg\| \frac{P_kP^\perp f}{|x|^2} \bigg\|_{2}^2   \\
& =  2\mathbf{R}_{n} \bigg\| \frac{f_*}{|x|}\bigg\|_{2}^2 + 2\bigg\langle  -\sum_{j=1}^n L_j^2 f_*, f_* \bigg\rangle + \bigg\| \Delta_r f + \mathbf{R}_{n} \frac{f}{|x|^2}   \bigg\|_{2}^2
\end{split}
\end{align}
from which we clearly obtain \eqref{e:remainderR} when $n \geq 4$ (since $\mathbf{R}_n \geq 0$). For $n = 2,3$ we follow a similar argument, except we use \eqref{e:MOW_leftv2} rather than \eqref{e:MOW_left}.
\end{proof}

\section{Further results and remarks} \label{section:final}

\subsection{Alternative derivation of $\mathbf{C}_3 = \frac{64}{25}$ in \eqref{e:sphericalcomparison}} \label{section:alternative}
Here we give a different proof of the estimate
\begin{equation} \label{e:sphericalcomparison3D}
\bigg\| \sum_{j=1}^3 L_j^2f \bigg\|_{L^2(\mathbf{R}^3)}^2 \leq \frac{64}{25} \|\Delta f\|^{2}_{L^2(\mathbf{R}^3)}
\end{equation}
by modifying the argument in \cite[Corollary 1]{e-l}. We begin by using $\Delta=\Delta_r+\frac{1}{|x|^2}\Delta_{\mathbb{S}^{2}}$ and \eqref{e:sumL_j2} to expand 
\[
\|\Delta f\|^{2}_{L^2(\mathbf{R}^3)} = \|\Delta_r f\|^{2}_{L^2(\mathbf{R}^3)} + \bigg\| \sum_{j=1}^3 L_j^2f \bigg\|_{L^2(\mathbf{R}^3)}^2 + 2 \mathrm{Re}\bigg\langle \Delta_r f, \sum_{j=1}^3 L_j^2f  \bigg\rangle.
\]
For the cross term, we use \cite[Eq. (2.15)]{e-l} (which relies on integration by parts and \eqref{e:1DHardy}) to estimate from below
\[
2 \mathrm{Re}\bigg\langle \Delta_r f, \sum_{j=1}^3 L_j^2f \bigg\rangle \geq -\frac{3}{2} \bigg\langle \frac{f}{|x|^2}, \sum_{j=1}^3 L_j^2f \bigg\rangle 
\]
and therefore
\[
2 \mathrm{Re}\bigg\langle \Delta_r f, \sum_{j=1}^3 L_j^2f  \bigg\rangle \geq -\frac{3}{4} \bigg( \frac{1}{\delta} \bigg\| \frac{f}{|x|^2} \bigg\|_{L^2(\mathbf{R}^3)}^2  + \delta  \bigg\| \sum_{j=1}^3 L_j^2f   \bigg\|_{L^2(\mathbf{R}^3)}^2 \bigg)
\]
holds for any $\delta > 0$. Using this, along with \eqref{e:radialRellich}, we obtain
\[
\bigg(1 - \frac{3\delta}{4} \bigg)\bigg\| \sum_{j=1}^3 L_j^2f \bigg\|_{L^2(\mathbf{R}^3)}^2
\leq
\bigg(\frac{3}{4\delta}  - \frac{9}{16}\bigg)\bigg\| \frac{f}{|x|^2} \bigg\|^2_{L^2(\mathbf{R}^3)} +
\|\Delta f\|^{2}_{L^2(\mathbf{R}^3)}.
\]
Up to here we have followed the argument in \cite[Corollary 1]{e-l}. At this point, the argument proceeds in \cite{e-l} by using the classical Rellich inequality \eqref{Rellich1}\footnote{Indeed, using \eqref{Rellich1} and rearranging, we obtain
\[
\bigg\| \sum_{j=1}^3 L_j^2f \bigg\|_{L^2(\mathbf{R}^3)}^2 \leq \frac{16}{3\delta(4 - 3\delta)} \|\Delta f\|^{2}_{L^2(\mathbf{R}^3)}
\]
for $0 < \delta < 4/3$. Optimizing in $\delta$ yields \eqref{e:sphericalcomparison3D} with constant $4$ on the right-hand side.}. However, we observe that in order to prove \eqref{e:sphericalcomparison3D} it suffices to consider $f$ satisfying $Pf = 0$ (since $\sum_j L_j^2 f = \sum_j L_j^2 P^\perp f$) and hence we may instead apply the refined version of the Rellich inequality \eqref{e:orthogonalR}. When $n = 3$ this reads
\begin{equation*} 
\bigg\|\frac{f}{|x|^2}\bigg\|_{L^2(\mathbf{R}^3)}^2 \leq \frac{16}{25}\|\Delta f\|_{L^2(\mathbf{R}^3)}^2\qquad\text{for}\quad Pf=0, 
\end{equation*}
and, as long as $\delta < \frac{4}{3}$, this yields
\[
\bigg\| \sum_{j=1}^3 L_j^2f \bigg\|_{L^2(\mathbf{R}^3)}^2 \leq \frac{16(3 + 4\delta)}{25\delta(4 - 3\delta)} \|\Delta f\|^{2}_{L^2(\mathbf{R}^3)}.
\]
An elementary calculus argument reveals that $\delta = \frac{1}{2}$ is the optimal choice, and this $\delta$ gives the desired estimate \eqref{e:sphericalcomparison3D}.

\subsection{An abstract version of Theorem \ref{t:spherical}} \label{section:abstract}
The argument in Section \ref{section:alternative} actually extends to the abstract setting considered in \cite[Corollary 1]{e-l} concerning the generalized Laplacian
\[
\mathcal{L} = \Delta_r + \frac{1}{r^2} \Lambda.
\]
Here, $-\Lambda$ is a non-negative, self-adjoint operator on $L^2(\mathbb{S}^{n-1})$ whose spectrum is purely discrete with isolated eigenvalues $\{ \lambda_k \}_{k \in \mathcal{I}}$ which may only accumulate at infinity\footnote{See \cite{e-l} (and also \cite{c-c-f}) for concrete examples of such operators.}. We also assume that zero is an eigenvalue of $-\Lambda$ and write $\lambda_0 = 0$. Then we are able to obtain the following.
\begin{theorem} \label{t:abstract}
Let $n = 2,3$. Then
\begin{equation} \label{e:sphericalcomparison3D_abstract}
\bigg\| \frac{1}{|x|^2} \Lambda f \bigg\|_{2}^2 \leq ( 1 - \mathbf{R}_n \sqrt{M_n})^2 \|\Delta f\|^{2}_{2},
\end{equation}
where
$
M_n = \max_{k \in \mathcal{I} \setminus \{0\}} (\lambda_k + \mathbf{R}_n)^{-2}.
$
Moreover, when $M_n = (\lambda_{k_0} + \mathbf{R}_n)^{-2}$ and $\lambda_{k_0} > -\mathbf{R}_n$, the constant in \eqref{e:sphericalcomparison3D_abstract} is best possible.
\end{theorem}
One can check that the constant in \eqref{e:sphericalcomparison3D_abstract} is consistent with Theorem \ref{t:spherical} in the sense that it becomes $\frac{64}{25}$ when $n = 3$ and infinite when $n = 2$. When $n \geq 4$ the inequality in \eqref{e:sphericalcomparison3D_abstract} was obtained in \cite[Corollary 1]{e-l} with constant $1$ and, by the obvious extension of the argument in our earlier proof of Theorem \ref{t:spherical}, this is best possible. 

\begin{proof}[Proof of Theorem \ref{t:abstract}] 
We consider the case $n = 3$. Since the argument for $n = 2$ is almost the same, we omit the details. 

To see \eqref{e:sphericalcomparison3D_abstract} we follow the argument in Section \ref{section:alternative} for the special case $\Lambda = \Delta_{\mathbb{S}^2}$, and simply replace use of \eqref{e:orthogonalR} by
\begin{equation} \label{e:orthogonalR_abstract}
\frac{1}{M_3} \bigg\|\frac{f}{|x|^2}\bigg\|_{L^2(\mathbf{R}^3)}^2 \leq \|\mathcal{L} f\|_{L^2(\mathbf{R}^3)}^2\qquad \text{whenever $\int_{\mathbb{S}^2} f(r,\theta) u_0(\theta) \, \mathrm{d}\sigma(\theta) = 0$ for all $r > 0$},
\end{equation}
where $u_0$ denotes a normalized eigenfunction of $-\Lambda$ corresponding to the eigenvalue $0$. The inequality \eqref{e:orthogonalR_abstract} follows from \cite[Theorem 1]{e-l}.

To see the optimality of the constant in \eqref{e:sphericalcomparison3D_abstract} when $M_3 = (\lambda_{k_0} - \tfrac{3}{4})^{-2}$ and $\lambda_{k_0} > \frac{3}{4}$, first note that in this case we have
\[
 1 + \frac{3}{4} \sqrt{M_3}  = \frac{4\lambda_{k_0}}{4\lambda_{k_0} - 3}.
\]
Next observe that for functions of the form
\[
f(x) = g(|x|) u_{k_0}(\tfrac{x}{|x|}),
\]
where $u_{k_0}$ is a normalized eigenfunction of $-\Lambda$ corresponding to the eigenvalue $k_0$,
the inequality $\| \frac{1}{|x|^2} \Lambda f \|_{2}^2 \leq C \|\Delta f\|^{2}_{2}$ is equivalent to 
\[
\lambda_{k_0} ( 2 + (\tfrac{1}{C} - 1) \lambda_{k_0} )  \int_0^\infty |g(r)|^2 \, \frac{\mathrm{d}r}{r^2} \leq  \int_0^\infty |g''(r)|^2 r^2 \,\mathrm{d}r  + 2(1 + \lambda_{k_0}) \int_0^\infty |g'(r)|^2 \,\mathrm{d}r.
\]
We now use the fact that, given any $R > 0$, the best constant in the inequality
\begin{align} \label{e:sphericallaststep_abstract}
\widetilde{C} \int_0^\infty |g(r)|^2 \, \frac{\mathrm{d}r}{r^2} 
\leq & \int_0^\infty |g''(r)|^2 r^2 \, \mathrm{d}r + R \int_0^\infty |g'(r)|^2  \, \mathrm{d}r
\end{align}
is given by $\widetilde{C} = \frac{1}{4}(\frac{1}{4} + R)$. This follows because the inequality \eqref{e:sphericallaststep_abstract} is equivalent to
\[
\int_0^\infty |h''(r)|^2 r^{2\nu + 2} \, \mathrm{d}r \geq (\widetilde{C} - \tfrac{1}{4}R(2-R)) \int_0^\infty |h(r)|^2 r^{2\nu - 2} \, \mathrm{d}r 
\]
via the relabeling $g(r) = r^\nu h(r)$ with $\nu = (\frac{R}{2})^{1/2}$. We then invoke the fact that the constant in \eqref{e:1DRellich} is best possible. 

From the above discussion we may conclude that
\[
\lambda_{k_0} ( 2 + (\tfrac{1}{C} - 1) \lambda_{k_0} ) \leq \tfrac{1}{4}(\tfrac{9}{4} + 2\lambda_{k_0}),
\]
or in other words $C \geq (\frac{4\lambda_{k_0}}{4\lambda_{k_0} - 3})^2$, and this yields the optimality of the constant in \eqref{e:sphericalcomparison3D_abstract}.
\end{proof}

\subsection{Sobolev space extension} \label{section:density}
The following lemma allows one to extend \eqref{e:orthogonalR} to all $f$ in $H^2(\mathbb{R}^n)$ for which $Pf = 0$ whenever $n \geq 3$. Here $H^2(\mathbb{R}^n)$ is the Sobolev space whose norm is given by
$
\|f\|_{H^2} = \| (1-\Delta) f \|_{2}.
$
Equivalently, we shall show that
\[
(\mathbf{R}_{n} + n-1)^2\bigg\|\frac{P^\perp f}{|x|^2}\bigg\|_{2}^2 \leq \|\Delta P^\perp f\|_{2}^2
\]
extends to all $f$ in $H^2(\mathbb{R}^n)$, as long as $n \geq 3$.
\begin{lemma} \label{l:density}
If $n \geq 3$ then $P^\perp C_0^\infty(\mathbb R^n\setminus\{0\})$ is dense in $P^\perp H^2(\mathbb R^n)$.
\end{lemma}

In the following proof of this lemma, we write $A \lesssim B$ to mean $A \leq CB$ for an appropriate constant $C$, and $\nabla^2$ denotes the Hessian. Also, we shall use the fact that $\Delta$ commutes with $P$; this can be seen, for example, on the  frequency domain using the fact that $P$ commutes with the Fourier transform and multiplication by radial functions.

\begin{proof}[Proof of Lemma \ref{l:density}]
Following the proof of in \cite[Proposition 9]{b-m-o}, we consider $\zeta_j \in C^\infty_0(\mathbb{R}^n \setminus \{0\})$ and $\rho_j \in C^\infty_0(\mathbb{R}^n)$ for $j \geq 1$ defined as follows. Take $C^\infty$ functions $\xi,\eta : \R \to [0,1]$ such that $\xi$ vanishes 
on $(-\infty,1/2]$ and coincides with $1$ on $[1,\infty)$, and $\eta$ coincides with $1$ on 
$(-\infty,1]$ and vanishes on $[2,\infty)$. For $j \geq 1$, 
we define $\zeta_{j}$ by 
\[
\zeta_{j}(x)=\xi(j|x|)\eta(\tfrac{|x|}{j}).
\]   
Then $\zeta_{j}$ is supported in the annulus $\{x\in\R^{n} : \frac{1}{2j}\le|x|\le2j\}$, 
and $\zeta_{j}(x)$ coincides with $1$ when $\frac{1}{j}\le|x|\le j$. Furthermore, we have
\begin{align} \label{e:zeta1}
|\nabla\zeta_{j}(x)|\lesssim
\begin{cases}
\frac{1}{|x|} |\xi'(j|x|)| &\text{if} \ 
|x|\in [\frac{1}{2j},\frac{1}{j}], \\
0 &\text{if} \ |x|\in [0,\frac{1}{2j}]
\cup [\frac{1}{j},j]\cup[2j,\infty], \\
\frac{1}{j}
&\text{if} \ |x|\in [j,2j].
\end{cases}
\end{align}
and 
\begin{align} \label{e:zeta2}
|\nabla^2\zeta_{j}(x)|\lesssim
\begin{cases}
\frac{j}{|x|} |\xi'(j|x|)| + j^2|\xi''(j|x|)| &\text{if} \ 
|x|\in [\frac{1}{2j},\frac{1}{j}], \\
0 &\text{if} \ |x|\in [0,\frac{1}{2j}]
\cup [\frac{1}{j},j]\cup[2j,\infty], \\
\frac{1}{j^2}
&\text{if} \ |x|\in [j,2j].
\end{cases}
\end{align}
Next, we define $\rho_{j}(x)=j^{n}\rho(jx)$ where $\rho : \R^n \to [0,\infty)$ is a 
$C^{\infty}$ function supported in the ball $\{x\in\R^{n} : |x|\le1/4\}$ and $\|\rho\|_{1}=1$. 
 
Take $f \in H^2(\mathbb{R}^n)$ and define $f_j = \rho_j * (\zeta_jP^\perp f)$. It suffices to prove that $P^{\perp}f_{j}\to P^{\perp}f$ in $H^{2}(\R^{n})$ as $j\to\infty$. To this end, we decompose
\begin{equation*}
P^\perp f_j - P^\perp f =  \rho_j * P^\perp f - P^\perp f -P(\rho_j * P^\perp f) - P^\perp(\rho_j * ((1-\zeta_j)P^\perp f)) 
\end{equation*}
Writing $g := (1-\Delta) f$, it follows that $P^\perp g \in L^2(\mathbb{R}^n)$ and thus
\[
\|\rho_j * P^\perp f - P^\perp f\|_{H^2} = \| \rho_j * P^\perp g - P^\perp g\|_2 \to 0.
\]
Also, by using the Cauchy--Schwarz and Minkowski inequalities, we have
\[
\|P(\rho_j * P^\perp f)\|_{H^2} = \| P(\rho_j * P^\perp g) \|_2 \leq \int_{|y| \leq 1/4} \rho(y) \| P^\perp g (\cdot - \tfrac{1}{j}y) - P^\perp g\|_2 \, \mathrm{d}y \to 0.
\]
For the remaining term in the decomposition, we first estimate
\[
\| P^\perp(\rho_j * ((1-\zeta_j)P^\perp f))  \|_{H^2} \leq \| \rho_j * ((1-\zeta_j)P^\perp f) \|_{H^2}
\]
and, since we clearly have $\| (1-\zeta_j)P^\perp f\|_2 + \| (1-\zeta_j)\Delta P^\perp f\|_2 \to 0$, it suffices to show
\begin{equation} \label{e:2noP}
 \lim_{j \to \infty} \| \nabla \zeta_j \cdot \nabla P^\perp f \|_2 = \lim_{j \to \infty}  \| \Delta \zeta_j P^\perp f \|_2 = 0.
\end{equation}
From \eqref{e:HardyRellich} we have
$
\int_{\mathbb{R}^n} |x|^{-2} |\nabla P^\perp f(x)|^2 \, \mathrm{d}x < \infty
$
so the dominated convergence theorem and the estimate \eqref{e:zeta1} quickly implies $\| \nabla \zeta_j \cdot \nabla P^\perp f \|_2 \to 0$.

For the remaining limit in \eqref{e:2noP}, we use the bound \eqref{e:zeta2} to get 
\begin{align*}
\| \Delta \zeta_j P^\perp f\|_2^2 \lesssim \int_{\mathbb{R}^n} \frac{1}{|x|^4} |\xi'(j|x|)P^\perp f(x)|^2 \, \mathrm{d}x + \int_{\mathbb{R}^n} \frac{1}{|x|^4} |\xi''(j|x|)P^\perp f(x)|^2 \, \mathrm{d}x + \frac{1}{j^4}\int_{\mathbb{R}^n} |P^\perp f(x)|^2 \, \mathrm{d}x
\end{align*}
and so, by the dominated convergence theorem, it suffices to check that $\int_{\mathbb{R}^n} |x|^{-4} |P^\perp f(x)|^2 \, \mathrm{d}x < \infty$.
To see this, we write
\begin{align*}
&\int_{\mathbb{R}^n} \frac{1}{|x|^4} |P^\perp f(x)|^2 \, \mathrm{d}x \\
& \qquad = \frac{1}{\sigma(\mathbb{S}^{n-1})^2} \int_0^{\infty} \int_{\mathbb{S}^{n-1}} \bigg| \int_{\mathbb{S}^{n-1}} (f(r\omega) - f(r\omega')) \, \mathrm{d}\sigma(\omega') \bigg|^2 \, \mathrm{d}\sigma(\omega) \, r^{n-5}  \mathrm{d}r.
\end{align*}
Therefore it follows that
\[
\int_{\mathbb{R}^n} \frac{1}{|x|^4} |P^\perp f(x)|^2 \, \mathrm{d}x \lesssim \int_{\mathbb{S}^{n-1}} \int_{1/2}^1  \int_{\mathbb{R}^n} \frac{1}{|x|^2} |\nabla f (tx + (1-t)|x|\theta)|^2 \, \mathrm{d}x \mathrm{d}t \mathrm{d}\sigma(\theta)
\]
and hence the Hardy inequality implies
\[
\int_{\mathbb{R}^n} \frac{1}{|x|^4} |P^\perp f(x)|^2 \, \mathrm{d}x \lesssim \int_{\mathbb{S}^{n-1}} \int_{1/2}^1 \int_{\mathbb{R}^n} |\nabla^2 f (tx + (1-t)|x|\theta)|^2 \, \mathrm{d}x \mathrm{d}t \mathrm{d}\sigma(\theta).
\]
Using the change of variables $x \mapsto tx + (1-t)|x|\theta$, it can be shown that
\[
\int_{\mathbb{S}^{n-1}} \int_{1/2}^1 \int_{\mathbb{R}^n} |\nabla^2 f (tx + (1-t)|x|\theta)|^2 \, \mathrm{d}x \mathrm{d}t \mathrm{d}\sigma(\theta) \lesssim \int_{\mathbb{R}^n} |\nabla^2 f(x)|^2 \, \mathrm{d}x < \infty
\]
as desired.
\end{proof}

\begin{remark}
For $n \geq 3$, the above density considerations allow us to extend the Rellich equality to $H^2(\mathbb{R}^n)$ under the constraint $Pf = 0$. The classical inequality \eqref{Rellich1} extends to all of $H^2(\mathbb{R}^n)$ as long as $n \geq 5$, and therefore, as well as raising the value of the constant, adding the constraint $Pf=0$ also allows one to consider a larger range of $n$.
\end{remark}

\subsection{Non-existence of maximizers} \label{section:nonexistence}
We show that there are no non-trivial cases of equality for \eqref{e:orthogonalR} by re-visiting our proof of this inequality and examining when the dropped terms vanish. In particular, assuming that equality holds in \eqref{e:remainderR} it follows from \eqref{mix-equal} that
$
\Delta_r f  + \frac{\mathbf{R}_{n}}{|x|^2} f  = 0
$
and, via the identity
\[
\Delta_r f  + \frac{\mathbf{R}_{n}}{|x|^2} f = |x|^{1-\frac{n}{2}} \partial_r(|x|^{-1}\partial_r(|x|^{\frac{n}{2}}f))
\]
we obtain
\begin{equation} \label{e:psirep}
|x|^{-1} \partial_r(|x|^{\frac{n}{2}}f)(x) = \psi\Big(\frac{x}{|x|}\Big)
\end{equation}
for some function $\psi : \mathbb{S}^{n-1} \to \mathbb{C}$. Consequently,
\begin{equation} \label{e:contradiction}
 \frac{\partial_rf(x)}{|x|} + \frac{n}{2} \frac{f(x)}{|x|^2}  = \frac{\psi(\frac{x}{|x|})}{|x|^{\frac{n}{2}}}.
\end{equation}
The $L^2(\mathbb{R}^n)$ norm of the function on the left-hand side of \eqref{e:contradiction} is controlled by $\|\Delta f\|_{2}$, and this quantity is finite for $f \in H^2(\mathbb{R}^n)$. The fact that the second term is controlled by $\|\Delta f\|_{2}$ follows immediately from \eqref{e:orthogonalR}. For the first term, we use the so-called Hardy--Rellich type inequality
\begin{equation} \label{e:HardyRellich}
C_{n} \bigg\| \frac{\nabla f}{|x|} \bigg\|_{2} \leq \| \Delta f\|_{2}
\end{equation}
for $f \in H^2(\mathbb{R}^n)$, where $n \geq 3$ and $C_{n}$ is a strictly positive constant. This inequality goes back to \cite{t-z} for $n \geq 5$, and the extension to $n = 3,4$ can be found in \cite{Beckner} and  \cite{gh-m} (see also \cite{Cazacu}).

On the other hand, it is easy to see that the $L^2(\mathbb{R}^n)$ norm of the function on the right-hand side of \eqref{e:contradiction} is infinite, unless $\psi = 0$. Thus, it follows from \eqref{e:psirep} that $f$ is of the form $f(x) = |x|^{-\frac{n}{2}} \phi(\frac{x}{|x|})$ and, since $f \in H^2(\mathbb{R}^n)$, this forces $f$ to be zero.

\end{document}